\documentclass[12pt,reqno]{amsart}
\usepackage{amssymb,latexsym,amsmath}

\theoremstyle{plain}
\newtheorem{thm}{Theorem}
\newtheorem{cor}[thm]{Corollary}
\newtheorem{lem}[thm]{Lemma}
\newtheorem{fact}[thm]{Fact}

\newtheorem{obs}[thm]{Observation}
\newtheorem{ex}[thm]{Example}

\newcommand{\lab}[1]{\label{#1}}
 
\newcommand{\mc}[1]{\mathcal{#1}} 
\newcommand{\CO}[1]{}
\begin{document}

\title[Perspectivity]{Perspectivity in
complemented modular lattices and regular rings}

\author[C. Herrmann]{Christian Herrmann}
\address{Technische Universit\"{a}t Darmstadt FB4\\Schlo{\ss}gartenstr. 7, 64289 Darmstadt, Germany}
\email{herrmann@mathematik.tu-darmstadt.de}

\begin{abstract} 
Based on an analogue for systems of partial isomorphisms
between lower sections in a complemented modular lattice
we construct a series of terms (including inner inverse
as basic operation and providing descending chains)
 such that
 principal right ideals $aR \cong bR$ in a (von Neumann) regular 
ring $R$ are perspective if
the series becomes stationary. 
 In particular, this applies if
 $aR \cap bR$ is  of finite height in $L(R)$.
This is used to derive,  for existence-varieties
$\mathcal{V}$
of regular rings, equivalence of unit-regularity
and direct finiteness, both conceived as a  property shared by all members
of $\mathcal{V}$. 
\end{abstract}

\subjclass{06C20,  16E50}
\keywords{Complemented modular lattice, von Neumann regular ring, perspectivity,unit-regularity, existence-varieties}

\maketitle

\section{Introduction}
(Von Neumann) regular rings $R$
and complemented modular lattices are closely connected fields
since  the work of von Neumann cf. \cite{neu}
- with $R$ one associates its lattice $L(R)$ of principal right ideals.
\emph{Unit-quasi-inverses} $u$ of elements $a$ (i.e. unites $u$ such that  $aua=a$)
have been introduced by Ehrlich~\cite{ehr1,ehr2},
a ring being \emph{unit-regular} if each element admits
some unit-quasi-inverse (such rings are, in particular \emph{directly finite}:
$ab=1$ implies $ba=1$).
 Ehrlich also showed that a regular
ring $R$ is unit-regular if and only if for all idempotents $e,f$
one has $eR\cong fR$ implying $(1-e)R\cong (1-f)R$.
Handelman~\cite{hand2} added further equivalent conditions, one of them
being that 
$eR \cong fR$ implies $eR$ perspective to $fR$ in $L(R)$.
Perspectivity, regularity, and unit-regularity of elements in 
general rings have been intensively studied, see e.g. \cite{garg,lee,mary,niel}.

 The purpose of the present note
is to give a sufficient condition 
on  $aR\cong bR$ 
in a regular ring $R$  
granting that  $aR$ is  perspective to $bR$
(which holds if   $a$ has a  unit-quasi-inverse)
 and to show that this applies if $aR\cap bR$ is of finite height in $L(R)$.

Here,  establishing perspectivity relies
on calculations in $L(R)$, for convenience  done
in abstract complemented modular lattices
endowed with a system of isomorphisms between lower sections
requiring properties  present in the case of  isomorphisms
induced by  isomorphisms 
between principal right ideals. 
The principal result is a reduction process 
associating $e_{n+1} \leq e_n$ and $f_{n+1}\leq f_{n}$ with given $e_n,f_n$
such that 
$e_n$ is perspective to $f_n$
(and so $e_0$ perspective to $f_0$) if  $e_{n+1}\cap f_{n+1}=e_n \cap f_n$ or
if  $e_{n+1}$ is perspective to $f_{n+1}$.  
In $\aleph_0$-complete complemented  modular lattices
$e_0$ is perspective to $f_0$  if  the meet
of the $e_n$ is perspective to the meet of the $f_n$. 

If one considers regular rings endowed with an operation
of inner inversion,  termination of this reduction process 
after $n$ steps can be
captured by an identity. This is applied to study unit-regularity
in the context of \emph{existence varieties} $\mathcal{V}$ of regular rings, that is,  classes closed under
homomorphic images, direct products, and regular subrings.
It is shown that for such classes unit-regularity 
is equivalent to direct finiteness, both considered as a property
required for all members.
(Compare this to the result of  Baccella and Spinosa
\cite{bac} that a  semiartinian regular ring is unit-regular
if and only if all its homomorphic images are directly finite.)
Another  property shown equivalent to unit-regularity  is that $\mathcal{V}$
does not contain nonartinian subdirectly irreducibles, equivalently, 
if it is generated by artinians of bounded finite length. 
Having $\mathcal{V}$ generated by artinians is not sufficient
for unit-regularity
in view of the result,  established by Goodearl, Menal, and Moncasi \cite[Thm.~2.5]{good2},
 that free regular rings are residually artinian
(and, according to Herrmann and Semenova~\cite[Cor.~14]{hs}, even residually finite). These results further developed  work of Tyukavkin
\cite{tyu} 
 obtaining the ring of row and column finite matrices as
well as certain regular rings $R$ of endomorphisms
of vector spaces as 
homomorphic images of (regular) subrings of products of finite-dimensional
matrix rings over (skew-)fields. In particular,
this approach was basic for
the study of existence varieties of regular rings 
in Herrmann and Semenova \cite{hs} and of varieties of $*$-regular rings
which
are generated by their artinian members,  see Micol~\cite{flo}, Herrmann and Semenova~\cite{rep}, and \cite{arch2}.

Thanks are due to the referees for valuable hints  to related literature and
results as well as to necessary clarifications and
a possible extension. 

\section{Complemented modular lattices}
\subsection{Preliminaries}
We refer to Birkhoff~\cite{birk} and von Neumann~\cite{neu}.
A \emph{lattice} $L$ is a  set endowed with a partial order $\leq$
such that 
 any two elements $a,b$ have infimum and supremum
written as \emph{meet} $a\cap b$ and \emph{join} $a+ b$, respectively.
We also write $ab=a\cap b$ and apply the usual preference rules.
All lattices to be considered
will have smallest element $0$ and greatest element $1$.

For $u \leq v$ in $L$, the \emph{interval} $[u,v]=\{x\mid u \leq x \leq v\}$
is again a lattice with the inherited partial order and operations.
A lattice $L$ is \emph{modular}   if 
\[ b \leq a \Rightarrow a (b + c)= b +ac.\]
Then the maps $x \mapsto x+b$ and $y \mapsto ay$
are mutually inverse isomorphisms between $[ab,a]$ and $[b,b+a]$.
An element $u$ of a modular lattice  $L$ is \emph{neutral} if 
$(u+x)(u+y)=u+xy$ for all $x,y \in L$; the set of
neutral elements is a sublattice of $L$.  
An element $a$ of a modular lattice  $L$ is of \emph{height} $d$ if
some (equivalently: each) maximal chain in $[0,a]$ has 
\emph{length} $d$, that is $d+1$-elements.
For the following see \cite[Ch. III Thm. 15]{birk}.

\begin{fact} \lab{f1}
In a modular lattice,  the direct product $[ab, a]\times [ab,b]$
embeds into $[ab,a+b]$ via $(x,y) \mapsto x+y= (x+b)(y+a)$.
\end{fact} 

 If $ab=0$ then we write $a + b=a\oplus b$.
If $a\oplus b=1$ then $b$ is a \emph{complement} of $a$.
A lattice $L$ is \emph{complemented} if each element admits some complement.
If $L$ is, in addition, modular then we speak of a CML.
In a CML each interval $[u,v]$   is again a CML
(within $[u,v]$, a complement of $x$ is given by $yv +u=(y+u)v$
where $x\oplus y=1$ in $L$).

In a CML,  elements $a,b$    are \emph{perspective}, written as $a\sim b$, 
 if they have a common complement;
equivalently, $a\sim_c b$ for some $c$, the latter meaning that
 $a+b=a+c=b+c$ and $ab=ac=bc$.
 We write $a \approx_c$ if $a\sim_c b$ and $ab=0$;
also $a \approx b$ if $a\approx_c b$ for some $c$.
 Applying  Fact~\ref{f1}
one obtains the following. 
\begin{fact}\lab{f2}
In a modular lattice, if $a_i \sim_{c_i} b_i, i=1,2$ and $a_1b_1a_2b_2\geq (a_1+b_1)(a_2+b_2)$  then $a_1+b_1 \sim_{c_1+c_2} a_2+b_2$.
\end{fact}

\begin{fact}\lab{f3a}
In a modular lattice one has $a\sim b$ if and only if $x \sim y$ 
for some (equivalently: all) $x,y$ such that $a=x \oplus ab$ and
$b=y \oplus ab$, Moreover, 
one has $a\oplus y =a+b=b\oplus x$ for such $x,y$.
\end{fact}
\begin{proof}
Observe that for such $x,y$ one has $ay = aby=0$
whence  by modularity 
$a(x+y) = x$ and, similarly, $bx=0$ and $b(x+y)=y$.   
By modularity  it follows that $ab(x+y)=0$. Now one has $ab+x+y= a+b$
so that   the map $z \mapsto z+ab$ is an isomorphism of
$[0,x+y]$ onto $[ab,a+b]$. Moreover, $a+y=a+ab+y=a+b$. Thus $a\oplus y=a+b$ and, similarly, $b\oplus x=a+b$.
\end{proof}

\begin{figure} 
\setlength{\unitlength}{8mm}
\begin{picture}(1,8)

\put(0.5,-0.2){$0$}

\put(-4.5,2){$y$}

\put(4.2,2){$v$}

\put(-.7,1.8){$x$}

\put(.5,1.8){$u$}

\put(-7.5,3.8){$z=x\oplus y$}

\put(5.2,3.8){$w=u\oplus v$}

\put(.3,0.8){$xu=zw$}
\put(-0.4,3.6){$y+v$}

\put(0,0){\circle*{0.2}}

\put(-4,2){\circle*{0.2}}

\put(4,2){\circle*{0.2}}

\put(0,5){\circle*{0.2}}

\put(0,4){\circle*{0.2}}
\put(4,3){\circle*{0.2}}
\put(-4,3){\circle*{0.2}}
\put(0,1){\circle*{0.2}}

\put(-5,4){\circle*{0.2}}

\put(-3,4){\circle*{0.2}}

\put(-4,5){\circle*{0.2}}

\put(5,4){\circle*{0.2}}

\put(3,4){\circle*{0.2}}

\put(4,5){\circle*{0.2}}

\put(-1,2){\circle*{0.2}}

\put(1,2){\circle*{0.2}}

\put(0,3){\circle*{0.2}}

\put(-1,6){\circle*{0.2}}

\put(1,6){\circle*{0.2}}

\put(0,7){\circle*{0.2}}

\put(0,0){\line(2,1){4}}

\put(0,1){\line(2,1){4}}

\put(0,0){\line(-2,1){4}}

\put(0,1){\line(-2,1){4}}

\put(-4,3){\line(2,1){4}}

\put(-4,2){\line(2,1){4}}

\put(4,3){\line(-2,1){4}}

\put(4,2){\line(-2,1){4}}

\put(0,0){\line(0,1){1}}

\put(-4,2){\line(0,1){1}}

\put(4,2){\line(0,1){1}}

\put(0,4){\line(0,1){1}}

\put(0,5){\line(1,1){1}}

\put(0,5){\line(-1,1){1}}

\put(-1,6){\line(1,1){1}}

\put(1,6){\line(-1,1){1}}

\put(-4,3){\line(1,1){1}}

\put(-4,3){\line(-1,1){1}}

\put(-5,4){\line(1,1){1}}

\put(-3,4){\line(-1,1){1}}

\put(4,3){\line(1,1){1}}

\put(4,3){\line(-1,1){1}}

\put(5,4){\line(-1,1){1}}

\put(3,4){\line(1,1){1}}

\put(0,1){\line(1,1){1}}

\put(0,1){\line(-1,1){1}}

\put(-1,2){\line(1,1){1}}

\put(1,2){\line(-1,1){1}}

\put(0,5){\line(1,1){1}}

\put(0,5){\line(-1,1){1}}

\put(-1,6){\line(1,1){1}}

\put(1,6){\line(-1,1){1}}

\put(-1,2){\line(-2,1){4}}

\put(1,2){\line(-2,1){4}}

\put(0,3){\line(-2,1){4}}

\put(-1,2){\line(2,1){4}}

\put(1,2){\line(2,1){4}}

\put(0,3){\line(2,1){4}}

\put(3,4){\line(-2,1){4}}

\put(5,4){\line(-2,1){4}}

\put(4,5){\line(-2,1){4}}

\put(-5,4){\line(2,1){4}}

\put(-3,4){\line(2,1){4}}

\put(-4,5){\line(2,1){4}}

\end{picture}
\caption{Lemma 4}\label{fig1}
\end{figure}

\subsection{Two lemmas on modular lattices}

\begin{lem}\lab{l1}In a 
modular lattice, if $z=x\oplus y$, $w=u\oplus v$, and $zw=xu$
then $yv=0$. If, in addition, 
$x\sim u$ and $y \sim v$ then also $z \sim w$. 
\end{lem}
\begin{proof}
Clearly, $yv\leq yzw \leq yx=0$.  Moreover, by modularity one has 
$(x+w)z=x+wz=x$ whence $(x+w)y=(x+w)zy =xy=0$
and, similarly, $(u+z)v=0$. Now, by Fact~\ref{f1}
$(x+w)(u+z)= x+u$  it follows
$(y+v)(x+u)= (y+v)(x+w)(u+z)= [y(x+w) +v)](u+z) =
v(u+z)=0 $ by modularity.  In particular, this implies
$xu(y+v)=0$ establishing the isomorphism $r \mapsto r+xu$ 
of $[0,y+v]$ onto $[xu,y+xu+v]$. 
Thus, if  $y\sim v$ then one has also  $y+xu \sim v+xu$.
Assuming that, in addition, $x \sim u$
one derives $z \sim w$ by Fact~\ref{f2}
since  $(x+u)(y+xu+v)= (x+u)(y+v)+xu =xu$.
\end{proof}

\begin{figure} 
\setlength{\unitlength}{8mm}
\begin{picture}(8.5,8)

\put(0,5.3){$a+d$}

\put(4.7,5.3){$b+d$}

\put(4.4,3){$c$}

\put(4.4,3){$c$}

\put(4.15,3.2){\circle*{0.2}}

\put(3,1.5){\line(2,3){1.18}}

\put(3.7,4.5){\line(1,-3){.48}}

\put(3.7,4.5){\circle*{0.2}}

\put(3.7,5.5){\circle*{0.2}}

\put(3,4){\circle*{0.2}}

\put(3.7,4.5){\line(0,1){1}}
\put(3,3){\line(0,-1){3}}
\put(3,1.5){\circle*{0.2}}

\put(3,7){\line(1,-2){.72}}
\put(3,6){\line(1,-2){.72}}
\put(3,3){\line(1,2){.72}}
\put(3,4){\line(1,2){.72}}

\put(3,4){\circle*{0.2}}


\put(0,3){\circle*{0.2}} 
\put(1.5,5.5){\circle*{0.2}}
\put(4.5,5.5){\circle*{0.2}} 
 
\put(4.5,5.5){\line(-1,1){1.5}}
\put(1.5,5.5){\line(1,1){1.5}}
\put(2.6,3.6){$d$}
\put(3.2,2.8){$u$}

 \put(-.5,2.7){$a$} 
\put(6,3){\circle*{0.2}}  
\put(6.2,2.7){$b$}  
\put(3,3){\line(0,1){1}}
\put(3,4){\line(1,1){1.5}}
\put(3,4){\line(-1,1){1.5}}
\put(1.5,4.5){\line(0,1){1}}
\put(4.5,4.5){\line(0,1){1}}
\put(3,6){\line(0,1){1}}
 \put(3,7){\circle*{0.2}} 
\put(1.5,1.5){\circle*{0.2}}    
\put(3,3){\circle*{0.2}}
\put(4.5,1.5){\circle*{0.2}}  
\put(3,0){\circle*{0.2}}
\put(1.5,4.5){\circle*{0.2}}
\put(4.5,4.5){\circle*{0.2}}
\put(3,6){\circle*{0.2}}
\put(0,3){\line(1,1){3}}
\put(0,3){\line(1,-1){3}} 
\put(1.5,1.5){\line(1,1){3}} 
\put(3,0){\line(1,1){3}} 
\put(1.5,4.5){\line(1,-1){3}}
\put(3,6){\line(1,-1){3}}
\end{picture}
\caption{Lemma 5}\label{fig2}
\end{figure}

\begin{lem}\lab{l2}
In modular lattice, $L$,  $a\sim b$ and  $d=(a+d)(b+d)$
jointly imply $a+d\sim b+d$.
\end{lem}
\begin{proof}  
Assume $a\sim_c b$ in $L$
and let $S$ denote  the sublattice  $S$ generated 
 by
 $a,b,c,d$. 
Let  $D_2$ and $M_3$ denote the $2$-element lattice and the height $2$ 
lattice with $3$ atoms, respectively. 
Obviously. $S$ is also generated  by $c$,
and the two  chains  $a \leq a+d$ and $b\leq b+d$.
Thus, by \cite{hkw} $S$ is a sub direct product
of lattices $D_2$, $M_3$, and lattices  $M$ 
where (the images of)  $a, a+d,b,b+d$ generate a boolean sublattice
with $(a*b)(a+d)(b+d)=0$ and $a+d+b=1$
and where (the image of) $c$ is a common complement of
the ``atoms''
$d=(a+d)(b+d)$, $a$, and $b$. 
Since also  $a\sim_c b$, 
$M$ is easily  seen to be trivial in both cases..

Thus, $S$ is a subdirect product of factors  $D_2$ and $M_3$, only.
Due to the given relations,    in any of these
factors the images of $d$ and $a+b$ take value
$0$ or $1$, only;
this means that $a+b$ and $d$  are neutral elements  of $S$.
It follows, that $u=d(a+b)$ is neutral, too, 
whence $a\sim_c b$ implies $a+u\sim_{c+u} b+u$ and this in turn $a+d \sim_{c+d} b+d$ 
via the isomorphism of $[u,a+b]$ onto $[d,1]$.
\end{proof}

\subsection{Partial isomorphisms and perspectivity}\lab{part}\lab{perp}
Motivated by the case where $L$ is the lattice of principal right ideals
of a regular ring, for any CML $L$ we consider
 lattice  isomorphisms $\alpha: [0,e] \to [0,f]$,
shortly written as $\alpha:e \to f$. 
For $g=e\cap f$ one has both $\alpha^{-1}(g)$ and $\alpha(g)$ 
defined. 
If $e'\leq e$ let $\alpha_{|e'}$
denote the restriction of $\alpha$ to $[0,e']$,
that is $\alpha_{|e'}:e' \to f'$, $f'=\alpha(e')$. 

\begin{obs}\lab{obs}
 $\alpha^{-1}(g)\cap \alpha(g)\leq g$ and  equality holds  if and only if $\alpha(g)=g$. 
\end{obs}
\begin{proof}
Clearly,
 $\alpha^{-1}(g)\cap \alpha(g)\leq \alpha^{-1}(f)\cap \alpha(e)= e \cap f =g$. 
Also, if $\alpha(g)=g$ then $\alpha^{-1}(g)=g$ 
whence  $\alpha^{-1}(g)\cap \alpha(g)=g$.
Conversely, if the latter holds then
$g \leq \alpha(g)$ and  $\alpha(g) \leq 
\alpha(\alpha^{-1}(g)) =g$ and it follows $g=\alpha(g)$.  
\end{proof}
We introduce a reduction process which yields perspectivity provided
that it stops after finitely many steps.
With $g=e\cap f$ one has $\alpha^{\#}:\alpha^{-1}(g)\to \alpha(g)$ 
defined by $\alpha^{\#}(x)=\alpha^2(x)$.

We consider \emph{admissible} systems $A$ of  isomorphisms $\alpha:e \to f$
 requiring the following axioms:
\begin{itemize}
\item[(A1)] If $\alpha:e \to f$ is in $A$ and  $e'\leq e$
then $\alpha_{|e'}$ is in $A$. 
\item[(A2)] If $\alpha:e \to f$ is in $A$ and $e\cap f=0$ then
$e \approx f$.
\item[(A3)] If $\alpha:e \to f$ is in $A$ and $g=e\cap f$
then $\alpha^{\#}$ is in $A$.
\end{itemize}

\begin{lem}\lab{l5} Consider  $\alpha:e\to f$  in admissible $A$.
\begin{itemize}
\item[(i)] If $\alpha(e \cap f)= e\cap f$ then $e \sim f$.
\item[(ii)] If $\alpha^{-1}(e\cap f) \sim \alpha(e\cap f)$ then $e \sim f$.
\end{itemize}
\end{lem} 
\begin{proof} Let $g=e \cap f$.
In (i) choose $y$ so that  $y\oplus g=e$ and put $v=\alpha(y)$.
Then $f= \alpha(e)= \alpha(y \oplus g)= \alpha(y)\oplus \alpha(g)
=v \oplus g$. With $x=u=g$ in Lemma~\ref{l1} it follows
$yv=0$ whence $y\sim v$ by axioms  (A1) and  (A2);  moreover,    
 $e \sim f$ by Lemma~\ref{l1}, again.

(ii): Put $x=g+\alpha^{-1}(g)$ and $u=g+\alpha(g)$, observe   that $u=\alpha(x)$. Observe that    $g\leq x \leq e$ and $g\leq u \leq f$
whence $x\cap u=g$. Thus,
 by 
hypothesis and Lemma~\ref{l2} 
one has $x\sim u$.  Choose $y$ such that $y\oplus x=e$ and put $v:=\alpha(y)$. 
Then $f=\alpha(x\oplus y) =\alpha(x) \oplus v$
and we conclude $y \cap v= 0$ by Lemma~\ref{l1}
whence $y \sim v$ by axioms (A1) and (A2);
finally, $e \sim f$ by Lemma~\ref{l1}.
\end{proof}

Given $\alpha:e \to f$ in  admissible $A$ define, by induction, $\alpha_0=\alpha$,
$e_0=e$, $f_0=f$
and $\alpha_{n+1}=\alpha_n^{\#}:e_{n+1}\to f_{n+1}$ . 
Put $g_n=e_n\cap f_n$ and observe that $e_{n+1}\leq e_n$ and
$f_{n+1} \leq f_n$ whence also $g_{n+1} \leq g_n$.

\begin{thm}\lab{main}
Given a CML, $L$, an admissible system $A$ of partial
isomorphisms in $L$, and $\alpha:e \to f$ in $A$,
one has  $e \sim f$
provided that $e_m \sim f_m$ for some $m$.
In particular, this applies 
if $\alpha_m(g_m)=g_m$ respectively $g_{m+1}=g_m$ for some $m$
or if  $e\cap f$ is of finite height in $L$.
\end{thm}
\begin{proof} 
 If   $e_{m}\sim f_{m}$ (which by (i) of Lemma~\ref{l5} is the case if
$\alpha_m(g_m)=g_m$) then $e_{m-1} \sim f_{m-1}$ by (ii) of Lemma~\ref{l5}
and it follows $e\sim f$ by induction. 
Now, assume $e\cap f$ of finite height in $L$.  By Observation~\ref{obs},
$g_{m+1} < g_m$
 unless $\alpha_m(g_m)=g_m$; thus, the latter has to occur for some $m$.
 \end{proof}

We now give for the case $g_{m+1}=g_m$ a proof with a
single application of Lemma~\ref{l2}, only.
A sequence $a_0, \ldots, a_m$  
in a modular lattice is \emph{independent}, if for all $n<m$,
 $a_n \sum_{n<k\leq m} a_k=0$ ; 
 equivalently
if $(\sum_{n \in I}a_n)(\sum_{n \in J}a_n)=0$ for all
$I,J \subseteq \{0,\ldots ,m\}$  such that $I\cap J=\emptyset$.
Induction using Fact~\ref{f2} yields the following.

\begin{fact}\lab{f5a}
In a modular lattice, if $a_0+b_0,a_1+b_1, \ldots, a_m+b_m$ 
is independent and $a_n \approx b_n$ for all $n\leq m$
then $\sum_{n\leq m} a_n \approx \sum_{n\leq m} b_n$. 
\end{fact}

\begin{lem}\lab{ind}
Consider $\alpha_n:e_n\to f_n$ and $g_n$, $n\leq m+1$ as above
and, for $n\leq m$,
 $x_n$ such that $e_n= x_n \oplus (e_{n+1}+g_n)$ 
and $y_n=\alpha_n(x_n)$.
Then the following hold.
\begin{itemize}
\item[(i)] $f_{n}= y_n \oplus(g_n+ 
f_{n+1}) $ and $(e_{k}+f_{k})(x_n+y_n)=0$ for all $k>n$.
\item[(ii)] $x_n\approx y_n$ for all $n \leq m$.
\item[(iii)]   $x_0,y_0, \ldots ,x_m,y_m$ is an independent sequence.
\item[(iv)]
$x:=\sum_{k\leq m}x_k \approx y:=\sum_{k\leq m}y_k$.
\item[(v)] If $g_m=g_{m+1}$ then  $x+g_0=e_0$, 
 $y+g_o=f_0$, and $e_0\sim f_0$. 
\end{itemize}
\end{lem}

\begin{proof} 
 Recall that
$e_{n+1}=\alpha_n^{-1}(g_n) \leq e_n$ and  
$f_{n+1}=\alpha_n(g_n)=\alpha_n^2(e_{n+1})\leq f_n$
and observe that $e_n=x_n \oplus(
\alpha_n^{-1}(g_n) +g_n)$. It follows that
$f_{n}= \alpha_n(e_n)=  y_n \oplus (g_n +\alpha_n(g_n))=y_n \oplus(g_n+ 
f_{n+1}) $ for $n\leq m$. Observe that $x_n y_n 
\leq e_n  f_n =g_n$ whence $x_ny_n \leq x_n g_n=0$. 
Thus $x_n\approx y_n$ by axiom (A2).

Moreover,  modularity yields 
$(e_{n+1}+ f_{n+1}) (x_n+y_n) \leq (e_{n+1}+f_n) (x_n+y_n)
=(e_{n+1}+f_n)x_n +y_n 
= (e_{n+1}+f_n)e_n x_n +y_n = (e_{n+1} +f_n e_n)x_n +y_n
=(e_{n+1} +g_n)x_n+y_n =0+y_n=y_n  $
and so $(e_{n+1} +f_{n+1})(x_n+y_n) =
(e_{n+1} +f_{n+1})y_n = (e_{n+1} +f_{n+1})f_ny_n=
(e_{n+1}f_n +f_{n+1})y_n \leq (g_n +f_{n+1})y_n =0$. 
It follows for all $n< m$
\[   \left( \sum_{n<k \leq m} x_k+y_k\right)(x_n+y_n)\leq 
(e_{n+1}+f_{n+1})(x_n+y_n)=0,\]
proving (iii); (iv) follows by Fact~\ref{f5a}. 
Dealing with (v), observe that $e_{m+1}=g_m$ whence $g_0+x_m\geq e_m$. Now,
backward  induction  yields $g_0+\sum_{m\geq k\geq n}x_k\geq e_n$
for all $n\leq m$, whence $g_0+x=e_0$. Similarly,
one has $g_0+y=f_0$
and  it follows $e_0\sim f_0$ by Lemma~\ref{f2}.

\end{proof}

\subsection{$\aleph_0$-complete complemented modular lattices}
A CML is $\aleph_0$-\emph{complete} if supremum $\sum_n a_n$
and infimum $\prod_n a_n$ exist for all families
$(a_n\mid n<\omega)$. According to Amemiya and Halperin \cite[Thm.9.5]{ame}
any such is also $\aleph_0$-\emph{continuous}, i.e.
$b  \sum_n a_n=\sum_n ba_n$ (resp. $b+\prod_n a_n= \prod_n (b+a_n)$)
if $(a_n\mid n<\omega)$ is upward (resp. downward) directed. 
A sequence $(a_n\mid n<\omega)$ is \emph{independent}
if each of its finite subsequences is independent.

\begin{fact}\lab{f5}
In an $\aleph_0$-complete CML, if $a_0+b_0,a_1+b_1, \ldots$ 
is independent and $a_n \approx b_n$ for all $n< \omega$
then $\sum_n a_n \approx \sum_n b_n$. 
\end{fact}
\begin{proof} Suppose $a_n \approx_{c_n} b_n$ for $n<\omega$
and write $x_n^+ =\sum_{m\leq n}x_m$
and $x_\omega= \sum_{n<\omega}x_n =\sum_{n<\omega}x_n^+$ for any  sequence $x_0,x_1,\ldots$.
By  Fact~\ref{f5a} one has $a_n^+ \approx_{c^+_n} b_n^+$
for all $n$. Since  sequences $x_0^+,x_1^+, \ldots$ are upward directed,
$\aleph_0$-continuity yields  \[a_\omega b_\omega = (\sum_{n<\omega} a_n^+)b_\omega
=\sum_{n<\omega} a_n^+b_\omega
=\sum_{n<\omega} a_n^+\sum_{m< \omega} b_m^+=\]
\[=\sum_{n<\omega} \sum_{m< \omega} a_n^+b_m^+ 
\leq \sum_{n,m<\omega} a^+_{\max{(n.m)}} b^+_{\max{(n,m)}} =0.\]
By symmetry one obtains $a_\omega c_\omega= b_\omega c_\omega=0$
while $a_{\omega}+b_{\omega} =a_\omega+c_\omega= b_\omega+c_\omega$
is obvious.
  \end{proof}

\begin{thm} 
Assume that   $\alpha:e \to f$
 is member of an admissible system $A$ of partial isomorphisms
of an $\aleph_0$-complete CML, $L$, and that $\prod_{n<\omega}e_n\sim
\prod_{n<\omega}f_n$ for $e_n,f_n$ defined as in Subsection~\ref{perp}.
Then it follows that $e \sim f$ in $L$.  
\end{thm}

\begin{proof}
Given $\alpha:e \to f$ in $A$ define $\alpha_n:e_n\to f_n$
as in Subsection~\ref{part} and $g_n=e_nf_n$.
Put $e_\infty=\prod_n e_n$, $f_\infty= \prod_n f_n$,
 and $g_\infty=\prod_n g_n =e_\infty f_\infty$ and recall that
$e_{n+1}=\alpha_n^{-1}(g_n) \leq e_n$ and  
$f_{n+1}=\alpha_n(g_n)=\alpha_n^2(e_{n+1})\leq f_n$.

Choose $x_n$ such that $e_n= x_n \oplus (e_{n+1}+g_n)$
and  $y_n=\alpha_n(x_n)$.
By Lemma~\ref{ind} one has
$f_{n}= \alpha_n(e_n)  =y_n \oplus(g_n+ 
f_{n+1}) $ and the
  sequence $x_0,y_0,x_1,y_1, \ldots $ is  independent; thus,
 Fact~\ref{f5} yields
$ x_\omega \approx y_\omega$
where  $x_\omega=\sum_n x_n$ and $y_\omega=\sum_n y_n$. 
On the other hand, again by Lemma~\ref{ind} and modularity,
\[   \left(e_\infty +f_\infty+ \sum_{n<k \leq m} x_k+y_k\right)(x_n+y_n)=0
\mbox{ for all } m> n
\]
that is $e_\infty+f_\infty, x_0+y_0, x_1+y_1, \ldots    $ 
is an independent sequence, too.
By hypothesis $e_\infty \sim f_\infty$ and
in view of $(x_\omega +y_\omega)(e_\infty +f_\infty)=0$
 Fact~\ref{f2} applies to yield $x_\omega+e_\infty \sim y_\omega+f_\infty$.

Now
 $ g+e_{n+1}+x_\omega \geq e_{n+1} + g_n +x_n=e_n$   
and by induction it follows $g+ e_n+x_\omega=e$ for all $n$.
Thus, by $\aleph_0$-continuity one has $g+x_\omega+e_\infty
= \prod_n (g+x_\omega+e_n) =e$. Similarly,  one obtains
$g+y_\omega+f_\infty=f$. Finally,
$e \sim f$ follows by Lemma~\ref{l2}.
\end{proof}

\section{Regular rings}\lab{reg}
\subsection{Preliminaries}
A  ring $R$ (associative and with unit) is (von Neumann)  \emph{regular} if for
each $a \in R$ there is a \emph{quasi-inverse}
or \emph{inner inverse} $x \in R$ such that $axa=a$;
equivalently, every right (left) 
principal ideal is generated by an idempotent, see Goodearl~\cite{good}
and Wehrung~\cite{fred}.
For a regular ring, $R$, the principal right ideals form a complemented sublattice $L(R)$
of the lattice of all right ideals; in particular, $L(R)$ is modular.
For artinian $R$, the height of $L(R)$ is the \emph{length} of $R$.

An element $a$ of $R$ is \emph{unit-regular} if there is  
a unit $u \in R$, a \emph{unit-quasi-inverse},  such that $aua=a$.
$R$ is \emph{unit-regular} if all its elements are unit-regular.
 Any such
ring is
\emph{directly finite} (that is $ab=1$ implies $ba=1$),
the converse not being true for regular rings, in general. 

If $e$ is an idempotent in a regular ring $R$,
then the \emph{corner} $eRe$ is a regular ring with unit $e$,
a homomorphic image of the regular subring
$eRe + (1-e)R(1-e)$ of  $R$.

Idempotents $e,f$ are Murray von Neumann \emph{equivalent} 
if $e=yx$ and $f=xy$ for some $x,y$.
For the following  see e.g. Handelman \cite{hand2} and Goodearl~\cite{good}..
\begin{itemize} 
\item[(1)] For any $a$ there is
a \emph{generalized} or \emph{reflexive inverse} $b$ such that $aba=a$ and $bab=b$,
e.g. $b=xax$ where $axa=a$. Then $ab$ and $ba$ are idempotents.
\item[(2)] $e,f$ are idempotents and  equivalent
if and only if $e=ba$ and $f=ab$ for some  $a,b$ as in (1).  Moreover, in this case
 $\omega_{a,b}(r)=ar$ defines an isomorphism
$\omega_{a,b}: bR=eR \to aR=fR$  of right $R$-modules 
with inverse $\omega_{b,a}$.
It follows that $fae=a$ and $ebf=b$.
\item[(3)] For idempotents $e,f$, 
 every $R$-module isomorphism $\omega:eR \to fR$
 is as in (2) where   $\omega(e)=a$ and $\omega^{-1}(f)=b$.
\item[(4)] If $eR \sim fR$ in $L(R)$ then $eR\cong fR$ as right $R$-modules.
If $eR \cap fR=0$ then the converse holds, too.
\item[(5)] If $c$ is another generalized inverse of $a$ 
then $bR \sim cR$ (being complements of $\{x \in R\mid ax=0\}$) 
and $x \mapsto cax$ is an isomorphism of $bR$ onto $cR$.
\end{itemize}
To prove (3), put $a=\omega(e)$ and $b=\omega^{-1}(f)$. 
Then one has  $\omega(r)= \omega(er) 
=\omega(e)r =ar$ for $r \in eR$ and, similarly, $\omega^{-1}(s)=bs$ for
$s \in fR$. 
Thus, $aba= \omega(ba)= \omega(\omega^{-1}(a))=a$ and, similarly,
$bab=b$.

A regular ring $R$ is {\em perspective} if
isomorphic direct summands of $R_R$ are perspective
in $L(R)$; equivalently,
 $aR$ is perspective to $bR$ for all
$aR\cong bR$ - for a more general result see Mary \cite[Thm.3.1]{mary}.

\begin{thm}\lab{hand} {\rm Handelman} \cite{hand2} A  regular
ring is unit-regular  if and only if it is perspective.
\end{thm}

An element $a$ of $R$ is \emph{strongly} $\pi$-\emph{regular}
if  there is $n$ such that $a^n \in a^{n+1}R\cap Ra^{n+1}$.
$R$ is \emph{strongly} $\pi$-\emph{regular} if so are all its elements.
\begin{thm}\lab{good3} {\rm Goodearl and Menal} \cite[Thm. 5.8]{good3}.
Strongly $\pi$-regular regular rings are unit-regular.
\end{thm}
In general rings, a strongly $\pi$-regular element is unit-regular
provided all its  powers are regular. For a detailed discussion and
proofs see Khurana~\cite{khu}

\subsection{Perspectivity}

The following ``local version''
of Handelman's Theorem should be well known.

\begin{lem}\lab{ext}
Given a generalized inverse $b$ of  $a$ in a regular ring $R$
one has idempotents  $e=ba$, $f=ab$,
and $g$ such that $gR=eR\cap fR$.
Now,  the following are equivalent.
\begin{itemize}
\item[(i)]
$fR$ and $eR$ are perspective in $L(R)$.
\item[(ii)] For some (all) idempotents $e',f'$ such that
$e'R\oplus gR=eR$ and $f'R\oplus gR=fR$  one has $e'R\cong f'R$. 
\item[(iii)] For some (all) idempotents $e',f'$ such that
$e'R\oplus gR=eR$ and $f'R\oplus gR=fR$ 
there is a 
unit $u$ of $R$ such that $aua=a$ and $e'R \cong f'R$ via $ue'$. 
\end{itemize}
 \end{lem}
\begin{proof}
In view of (2) in the preceding subsection,   $e=ba$ and $f=ab$
are idempotents. 
Consider $e',f'$ as in (ii) and (iii) and observe that such exist
since $L(R)$ is complemented.

  By Fact~\ref{f3a} we have $e'R \cap f'R=0$
and, moreover,  $eR \sim fR$ if and only if $e'R\sim f'R$.
By (4), the latter is equivalent to $e'R \cong f'R$. 
This proves that (i) is equivalent to (ii).

Now, assume (ii), in particular 
$f'R \cong e'R$  via some isomorphism $\omega'$.
Choose an idempotent $h$ such that $eR+fR=hR$.
Then \[eR\oplus f'R \oplus(1-h)R=R =fR\oplus e'R\oplus (1-h)R,\] again by Fact~\ref{f3a}.
In view of (2) define 
\[\omega(r+s+t)= \omega_{b,a}(r)+ \omega'(s)+ t \mbox{ for } r \in fR,
s \in e'R  \mbox{ and } t \in (1-h)R\]  
to obtain an automorphism of the right $R$-module $R=1R$.
 By (3) there are $u,v$ in $R$ such that $\omega=\omega_{u,v}$;
in particular,  $u$ is a unit and $v=u^{-1}$. Moreover, 
$ur= \omega_{b.a}(r) =br $ for $r \in fR$, in particular $ua=ba =e$.
Thus $aua=ae =a$, proving that (ii) implies (iii).

Finally, assume (iii). Then $u^{-1}f'=e'$ whence $x \mapsto
ux$ is an $R$-module  isomorphism of $e'R$ onto $f'R$ with inverse
$y \mapsto u^{-1}y$.  Thus, (ii) and (iii) are equivalent, too.
\end{proof}

For $A,B \in L(R)$ and right module isomorphism 
$\omega:A\to B$ one has the 
induced lattice isomorphism $\omega_L:[0,A]\to [0,B]$.
Let  $A(R)$ denote the set of all these. 
\begin{lem}\lab{ad}
$A(R)$ is an admissible system of partial isomorphisms of $L(R)$. 
\end{lem}     
 \begin{proof}
  Consider $\omega:A \to B$ in $A(R)$ and observe that $\omega_L(X)=
  \omega(X)$ for all $X \leq A$. Thus,
  if $A' \leq A$ in $L(R)$ then $(\omega_{|A'})_L:[0,A']\to
  [0,B']$ in $A(R)$ with $B'=\omega(A')\leq B$, proving axiom (A1).
 Similarly, for $C=A\cap B$, $A'=\omega^{-1}(C)$, and $B'=\omega(C)$
 one has $\omega_L^\# =(\omega_{|C} \circ \omega_{|A'})_L$ in $A(R)$,
 proving axiom (A3).
Finally,
 (A2) follows from (4).
\end{proof}

As observed by the referee, the following can be obtained
in purely  ring theoretic terms: it is an immediate
consequence of Lemma~\ref{ext} and \cite[Proposition 4.13]{good}.

\begin{cor}
For  $A,B$  in
the lattice $L(R)$ of principal right ideals of the regular ring $R$, if  $A\cap B$ is of finite height in $L(R)$
 then
$A,B$  are  perspective in $L(R)$ if and only if they are isomorphic
as $R$-modules. 
\end{cor}  
\begin{proof}
Assume $A\cap B$ of finite height.
If $A\cong B$ then there is a lattice
isomorphism $\alpha:[0,A]\to [0,B]$ in $A(R)$
and  in view  of Lemma~\ref{ad} and Theorem~\ref{main}
it follows that $A\sim B$. The converse follows from (4).
\end{proof}

\begin{cor} An element  $a$ in a regular ring $R$
is unit-regular provided
that there is a reflexive inverse $b$
of $a$ such that $bR \cap aR$ is of finite height in $L(R)$.
\end{cor}

\subsection{Regular rings with operation of  quasi-inversion}\label{var}
A regular ring  may be considered as an algebraic structure also 
endowed with an operation $a \mapsto a'$ of quasi-inversion.
The class $\mathcal{R}$ of all these structures is then defined
by the identities  for rings with unit together with 
$xx'x=x$. 
 As observed above, the term $x^+=x'xx'$ 
then yields a generalized inverse $a^+$ of $a$ and $\gamma(x)=x x^+$
yields idempotents $\gamma(a)$ such that $\gamma(a)R=aR$.
Though, in this setting, regular subalgebras (these algebras
endowed with an operation of  inner inversion)
 of unit-regular rings 
may fail to be unit-regular as shown by Bergman~\cite{berg}.
As remarked by the referee, 
unit-regularity can be equationally defined adding the 
identities $x'(x')'=1$ and $(x')'x'=1$.
See Ara, Goodearl, Nielsen. Pardo, and Perera~\cite{ara}
for a comprehensive introduction.

For the following see Wehrung~\cite[Lemma 8-3.12]{fred}.

\begin{lem}\lab{fred}
There are binary  terms
$x \vee y$, 
 $x \wedge y$, and $x\ominus y$ in the language of
$\mathcal{R}$ such that, for all $R\in \mathcal{R}$ and $a,b\in R$, 
$a \vee b$,
 $a\wedge b$, and $a \ominus b$ are
idempotent,
$(a\vee b)R=aR+bR$, 
$(a \wedge b)R=aR  \cap bR$, and $(a\ominus (a\wedge b))R\oplus (a\wedge b)R=aR$.\end{lem}

\begin{thm}\lab{mainr} For each natural number $n$ there are
binary terms $t_n(x,y)$, $u_n(x,y)$, and $p_n(x,y)$ in the language of $\mathcal{R}$
such that the following hold for all $R \in \mathcal{R}$ and mutually 
reflexive inverses $a,b \in R$:
$t_n(a,b)$ is idempotent;  moreover, if $t_{n+1}(a,b)t_n(a,b)=t_n(a,b)$ then
\begin{itemize}
\item[(i)]  $bR$ and $aR$    are
perspective in $L(R)$:   $bR \sim_{p_n(a,b)} aR$.
\item[(ii)]  $u_n(a,b)$ is a unit such that $au_n(a,b)a=a$.
\end{itemize} 
   If $R$ is of length at most $n+2$ then 
    $t_{n+1}(a,b)t_n(a,b) =t_n(a,b)$ 
for all mutually reflexive inverses $a,b$  in  $R$.
\end{thm}

\begin{proof} 
With idempotents $e_0=ba$ and $f_0=ab$ 
one has the isomorphism $\omega_{ab}:e_0R \to f_0$
given by $x  \mapsto ax$ inducing the
isomorphism  $\alpha=\alpha_0:[0,eR] \to [0,fR]$
given by $\alpha(xR)=axR$ with inverse $\alpha^{-1}(xR)= bxR$.
Recalling the construction in Subsection~\ref{part}
put $g_0= e_0 \wedge f_0$ and,
recursively, 
 \[g_n=e_n \wedge f_n,\,e_{n+1}=\gamma(b^{2^n}g_n),\;f_{n+1} =
\gamma(a^{2^n}g_n)\]
to obtain $\alpha_{n+1}:[0, e_{n+1}R] \to [0,f_{n+1}R]$
given by $\alpha_{n+1}(xR)= \alpha_n^2(xR)=a^{2^n}xR$.
Accordingly, put $t_0(x,y) = yx \wedge xxy$,
 and, inductively,
\[  t_{n+1}(x,y) =  y^{2^{n}} t_n(x,y) \wedge x^{2^{n}} t_n(x,y).  \]
Thus, for $a,b$ as above one has
$t_n(a,b)R= g_nR= e_nR\cap f_nR$ 
 whence
\[t_{n+1}(a,b)t_n(a,b) =t_n(a.b)\; \Leftrightarrow \; g_{n+1}R =g_nR.\]
 Thus,
supposing  $t_{n+1}(a,b)t_n(a,b) =t_n(a,b)$, 
  $bR$ and $aR$ are perspective in $L(R)$ by Theorem~\ref{main}
and  Lemma~\ref{ext} applies to provide the existence of a unit $u$ in $R$
such that $aua=a$ and idempotent $p \in R$ such that 
$bR\sim_{pR} aR$.
To prove the existence of  terms $u_n(x,y)$ and $p_n(x,y)$, as required,  it suffices 
to observe that all this  applies, in particular,
to $R$ being  the free algebra in  $\mathcal{R}$ with 
 generators $a,b$ and 
relations $aba=a$, $bab=b$, and $t_{n+1}(a,b) t_n(a,b)=t_n(a,b)$. 

Now,  assume that $g_k \neq g_{k+1}$ for all $k \leq m$.
Then one obtains a chain  $e_0R+f_0R> e_0R>g_0R> \ldots >
g_{m+1}R$  of length $m+3$ in $L(R)$.
Thus,
if $r$ is of length at most $n+2$ then  $g_m =g_{m+1}$ for some $m \leq n+2$ and it follows $g_k=g_m$
for all $k\geq m$, in particular $g_n=g_m=g_{n+1}$.
\end{proof}

\begin{ex}
\begin{itemize}
\item[(i)] There are
 $a,b,c$ in some  unit-regular ring $R$ such that
$a,b$ and $a,c$ are pairs of reflexive inverses,
$aR$, $bR$, and $cR$ pairwise perspective,
$t_{n+1}(a,b)t_n(a,b)\neq t_n(a,b)$ for all $n$, 
and $t_0(a,c)=0$.
\item[(ii)]
There are a regular ring $R$ and reflexive inverses $a,b$ 
in $R$ such that $t_0(a,b)=0$
but both $a$ and $b$ are not strongly $\pi$-regular,
\end{itemize}
\end{ex}
\begin{proof}
Considering (i) let $V$ a  vector space  of dimension $n+3$.
We show by induction that ${\sf End}(V)$
contains some $a,b$ with associated $g_n>g_{n+1}$.
More precisely, we show 
that for any  subspaces $V_1 \neq V_2$ 
of codimension $1$  
there is such $a$ with generalized inverse $a^+$
and  restricting to an isomorphism 
$V_1 \to V_2$ and such that
$V_1={\sf im }\, a^+$ and $V_2= {\sf im }\, a$. 
If  $n=0$ choose $v_i$ such that
$V_1\cap V_2={\sf span}\, v_3$ and $V_i= {\sf span}\,v_i +V_1\cap V_2$ 
for $i=1,2$. Define the endomorphism 
$a$ by $a(v_1)=v_2$, $a(v_2)=0$, and $a(v_1+v_3)=v_3$
and $a^+$ by $a^+(v_1)=0$, $a^+(v_2)=v_1$, and $a^+(v_3)=v_1+v_3$.
Proceeding from $n-1$ to $n$
 choose $W$ of codimension $1$ in $V$
such that $V_1 \cap V_2 \not\subseteq W$ and put $W_i=W\cap V_i$. 
Choose  endomorphisms $a_0,a_0^+$ of $W$
connecting $W_1$ and $W_2$ according to hypothesis.
Choose $v_3 \not\in W$ and $v_i$ such $V_i={\sf span}\,v_i +W_i$
for $i=1,2$  and extend $a_0$ and  $a_0^+$ 
to obtain $a$ and $a^+$, defined for $v_i$ as above.

By this construction there 
are finite-dimensional  $W_n=V_n\oplus U_n$
with mutally  reflexive inverses
$a_{0n},b_{0n}$ in $V_n$ such that $t_{n+1}(a_{0n},b_{0n}) 
t_n(a_{0n},b_{0n}) \neq t_n(a_{0n},b_{0n})$ 
 and isomorphism $c_{0n}:V_n\to U_n$.
Choose $a_n$ extending $a_{0n}$ and $c_{0n}^{-1}$,
and $b_n,c_n$ extending $b_{0n}$ and $c_{0n}$, respectively, such
that $b_n|U_n=0$ and $c_n|U_n=0$. 
Then the direct product of the ${\sf End}(W_n)$ provides 
$R$ and $a,b,c$ as required.

In (ii)
consider a vector space $V$ with basis $v_n,w_n (n \in \mathbb{N})$,
$R={\sf End}(V)$ and define $a(v_n)=w_n$, $a(w_n)=w_{n+1}$,
$b(w_n)=v_n$, and $b(v_n)=v_{n+1}$. 
\end{proof}

\subsection{(Existence) varieties of unit-regular rings }
Observe that subrings of regular rings are not regular, in general,
an obvious example being $\mathbb{Z} \subset \mathbb{Q}$.
Thus, to deal with classes $\mathcal{C}$ of regular rings
in the framework of Universal Algebra, without specifying
operations of quasi-inversion, it is convenient to introduce
the class operator ${\sf S}_\exists(\mathcal{C})$ 
  associating with $\mathcal{C}$ the class of all regular
rings which are subrings of members of $\mathcal{C}$.
Referring to the usual operators ${\sf H}$,
${\sf P}$, and ${\sf P}_u$ for homomorphic images, direct 
products and ultraproducts (which preserve regularity),
a class $\mathcal{V}$ of regular rings which is   closed under  
under ${\sf H}$, ${\sf S}_\exists$, and ${\sf P}$
(whence also ${\sf P}_u$) is an \emph{existence variety},
shortly $\exists$-\emph{variety} (cf. Hall \cite{hall} for this concept).
According to Herrmann and Semenova~\cite[Thm. 16]{hs}
every existence variety of regular rings is generated by its
 artinian members.

For a class $\mathcal{C}$ of regular rings 
 let ${\sf T}(\mathcal{C})$
consist of all regular rings endowed  with an operation of 
 inner inversion  (that is, members of $\mathcal{R}$ 
as defined in the previous subsection)
where the underlying  ring is in $\mathcal{C}$.
For a ring $R$ let $R^{n\times n}$ denote the ring of $n$-by-$n$-matrices.

\begin{fact}\lab{hs}\lab{hs2}
\begin{itemize}
\item[(i)]
The smallest existence variety ${\sf V}_\exists(\mathcal{C})$
containing $\mathcal{C}$ is ${\sf HS}_\exists{\sf P}(\mathcal{C})$.
\item[(ii)] $R \in  {\sf HS}_\exists{\sf P}_u(\mathcal{C})$
for every subdirectly irreducible $R\in {\sf V}_\exists(\mathcal{C})$.
\item[(iii)]  Any subdirectly irreducible regular ring $R$ is an $F$-algebra
for a suitable field $F$. Moreover, if such $R$ is nonartinian 
then $F^{n\times n} \in \sf{H}{\sf S}_\exists(R)$ for all $n<\omega$ and
${\sf V}_\exists(R) ={\sf V}_\exists\{F^{n\times n}\mid n<\omega\}$.
\item[(iv)]  Any identity in the
language of $\mathcal{R}$ which is valid in ${\sf T}(\mathcal{C})$
is also valid in  ${\sf TV}_\exists(\mathcal{C})$.
 \item[(v)] ${\sf TV}_\exists(\mc{C}) ={\sf VT}(\mc{C})$.
\end{itemize}
\end{fact}
\begin{proof}
Referring to Herrmann and Semenova \cite{hs},
(i), (iv), and (v) follow from 
 \cite[Prop. 10 (i)]{hs}.
(ii) is \cite[Prop. 7]{hs}.
(iii) follows from  \cite[Thm. 16]{hs} and its proof.
\end{proof}

Define   $s_n(x)= t_n(x,x^+)$. 
In the following, the equivalence of (5) and (9) is   due, in essence, 
to O'Meara and Raphael~\cite[Thm.2.15]{meara}.

\begin{thm}\lab{23}
For an existence variety $\mathcal{V}$ of regular rings the following are
equivalent (where the notion of ``term'' and the terms $x^+$ 
and $t_n(x)$ are as in Thm.~\ref{mainr}) 
\begin{enumerate}
\item All members of $\mathcal{V}$ are perspective.
\item All members of $\mathcal{V}$ are unit-regular.
\item All subdirectly irreducible  members of $\mathcal{V}$ are 
directly finite.
\item All subdirectly irreducible  members of $\mathcal{V}$ are 
artinian.
\item 
There is $d<\omega$ such that all artinian subdirectly irreducible  members of $\mathcal{V}$ are of length $\leq d$.
\item 
There are $d<\omega$ and a class  $\mathcal{C}$ of artinian regular rings
of length $\leq d$ such that $\mathcal{V}={\sf V}_\exists(\mathcal{C})$.
\item  There is $n<\omega$ such that $s_{n+1}(x)s_n(x)=s_n(x)$ 
is valid in ${\sf T}(\mathcal{V})$. 
\item There is $m<\omega$ such that
the identities $(x^{m+1})(x^{m+1})^+x^m=x^m$ and 
$x^m(x^{m+1})^+x^{m+1}=x^m$  are valid in ${\sf T}(\mathcal{V})$.   
\item There is a  term $u(x)$  
yielding unit inner inverses, uniformly in $\mathcal{V}$; that is,
$u(a)$ is a unit inner inverse for any $R\in \mathcal{V}$ and $a \in R$. 
\end{enumerate}
Actually, given $d\geq 2$ in (5) one can choose $n=d-2$
in (7)  and $m=d$ in (8).
\end{thm}
\begin{proof} 
(7) implies (1) by (i) of Theorems~\ref{mainr}.
(8) implies (1), too, in view of Theorem~\ref{good3}.
 (1) is equivalent to (2)
by Theorem~\ref{hand}, and (2) implies (3).

Each of (3) and (4) implies (5):
 Indeed, assume that
there are artinian subdirectly irreducibles  $R_n\in \mathcal{V}$
with no bound on length. Renumbering and  passing to corners and isomorphic copies,  we may assume that $R_{n}\cong D_n^{n\times n}$ for some division ring $D_n$.
Thus, for fixed $m$ and all $n\geq m$, the ring
  $R_n$ contains a subring $R_{mn} \cong D_n^{m\times m}$.
Choose $R_{mn}=0$ for $n<m$. Thus, in particular $R_{mn}\in \mathcal{V}$
for all $m,n$.
Recall that, for fixed $m$, the class of all  rings
isomorphic to $D^{m\times m}$ for some division ring $d$
can be finitely first-order axiomatized if
one adds $m^2$ constants for a system of matrix units.
Thus, choosing  a non-principal ultrafilter $\mathcal{F}$  
on $\mathbb{N}$ one has for any fixed $m$
 the ultraproduct $(\prod_{n\in \mathbb{N}}R_{mn})/\mathcal{F}$
isomorphic to $D^{m\times m}$ where 
$D=(\prod_{n\in \mathbb{N}} D_n)/\mathcal{F}.$
It follows $D^{m\times m} \in \mathcal{V}$
and thus $F^{m\times m} \in \mathcal{V}$
for all $m$ where $F$ is the center of $D$.
 Now, consider 
  any infinite-dimensional $F$-vector space $W$ and
 ${\sf End}(W_F)$; the latter is subdirectly irreducible, nonartinian,
and not directly finite. By Fact~\ref{hs}(iii) one has 
${\sf End}(W_F)\in \mathcal{V}$ contradicting both (3) and (4).

(5) implies  (4):
Assume there is subdirectly irreducible
$R \in \mathcal{V}$ which is not artinian. By Fact~\ref{hs}(iii),
$R$ is an $F$-algebra for some field $F$ and 
 $\mathcal{V} \supseteq {\sf V}_\exists(R) =  {\sf V}_\exists\{F^{n\times n}\mid n <\omega\}$ so that the (subdirectly irreducible)
 $F^{n\times n} \in \mathcal{V}$ 
for all $n<\omega$, contradicting (5).

(5) implies (6):
Since  (5) implies (4), in view of Fact~\ref{hs}(ii)
 it  follows that $\mathcal{V}$ is generated by members
 of length $\leq d$.

 (6) implies (7) and (8):
Let $\mathcal{R}_d$ consist of all artinian regular rings 
which are of length at most $d$.
Thus,
$\mathcal{V}\subseteq {\sf V}_\exists(\mathcal{R}_d)$.
 Now, consider subdirectly 
irreducible $R \in \mc{V}$.
By Fact~\ref{hs}(ii) one has 
    $R \in {\sf HS}_\exists{\sf P}_u(\mathcal{R}_d)$.
Since the property of having length $\leq d$
 can be expressed, easily,  by a first-order formula
(in various ways), we have  ${\sf P}_u(\mathcal{R}_d) \subseteq \mathcal{R}_d$
while  ${\sf HS}_\exists(\mathcal{R}_d) \subseteq \mathcal{R}_d$ is obvious.
This implies that 
$R \in \mathcal{R}_d$ whence $\mathcal{V} \subseteq {\sf V}_\exists(\mathcal{R}_d)$ by Fact~\ref{hs}(ii). 
 By (ii) of Theorem~\ref{main}
the identities $s_{n+1}(x)s_n(x)=s_n(x)$,
$(x^{m+1})(x^{m+1})^+x^m=x^m$, and
$x^mx^m(x^{m+1})^+x^{m+1}$  (where $n=d-2$ and $m=d$)
are valid in  ${\sf T}(\mathcal{R}_d)$ and so in ${\sf T}(\mathcal{V})$
by Fact~\ref{hs}(iv). 

(9) implies (2), trivially. Conversely, by (v) of   Fact~\ref{hs}, 
the free ${\sf T}(\mc{V})$-algebra  $A$
on a single generator $a$ is unit-regular; that is, there  
is a  term $u(x)$ that that $u(a)$ is a unit inner inverse of $a$
in $A$. Thus $u(b)$ is a unit inner inverse of
$R$ for all $R\in \mc{V}$ and $b \in R$
\end{proof}
The following has been   suggested   by the referee.
\begin{cor}\lab{24} Considering classes $\mc{V}$ closed under
${\sf H}$, ${\sf S}_\exists$, and ${\sf P}_u$ 
the equivalences of
Theorem~\ref{23} remain valid 
if  ${\sf P}$
is replaced by ${\sf P}_u$,
everywhere.
\end{cor} 
\begin{proof}
For the equivalence of (1), $\ldots$. (8) it suffices to adapt
(i)--(iv) of Fact~\ref{hs} and to observe that
 ${\sf P}_u{\sf H}(\mc{C})
\subseteq {\sf H}{\sf P}_u(\mc{C})$ 
which is well known due to the fact that an ultraproduct
of surjective homomorphisms amounts to a surjective homomorphism.
In order to complete the proof that both (3) and (4) imply (5)
observe that
 ${\sf End}(W_F)\in {\sf HS}_\exists 
{\sf P}_u(\{F^{m\times m}\mid m <\omega\})$  by (ii) of Fact~\ref{hs}.
and that $F^{m \times m} \in {\sf S}_\exists{\sf P}_u(\{R_n\mid n<\omega\})$
whence ${\sf End}(W_F) \in \mc{V}$. Also (9) implies (2), trivially.

To prove that  (9) implies (5) we proceed by contradiction. 
Assume that for each n there is  $R_n\in \mc{V}$
which is  subdirectly irreducible  artinian
of length $d_n \geq n$. In particular, there is
a  division ring $F_n$ such that 
 $R_n\cong M_{d_n}(F_n)$. Let $R$ a non-trivial ultraproduct  of the $R_n$.
$R$
contains, for each $n$, a $d_n\times d_n$ system of matrix units 
 whence a subring isomorphic to $M_{d_n}(F)$ where $F$ is the prime subfield
of the corresponding ultraproduct of the $F_n$.
According to O'Meara and Raphael~\cite[Thm. 2.15]{meara}
there is no $u$ as required in (9).

\end{proof}

\begin{cor}
The analogues of the  equivalences of Theorem~\ref{23}
are valid for varieties $\mc{V}$ of regular rings 
with inner inverse as basic operations, i.e. subvarieties of
$\mc{R}$ as defined in Subsection~\ref{var}; that is, with 
 ${\sf HSP}$ in place of
${\sf V}_\exists$ and omitting operator ${\sf T}$.
\end{cor} 
\begin{proof}
The only step in the proof of Thm.\ref{23}  to be reconsidered is $(5) \Rightarrow (4)$;
this follows as in the proof of Cor.\ref{24}.
\end{proof}

\noindent
{\bf Remark}.
At present,
only some of the implications could be extended 
  to  $*$-regular rings. This is related to the following questions, the 
first being due to Handelman. Are all $*$-regular rings directly finite or even
unit-regular? Is the  
variety   of all  $*$-regular rings (with  pseudo-inversion) generated by artinians? 
If a variety is generated by artinians does it have   all members directly finite?
The reasoning given in \cite{arch,arch2} for a positive answer to the latter
is incomplete and, most likely, cannot be completed.
Thus, direct finiteness of semiartinian $*$-regular rings 
(claimed in \cite{arch3})
remains open, too.

\section{Declarations}
\subsection{Funding} None.
\subsection{Data availability} Not applicable.
\subsection{Ethical standards} The author declares that there
are no conflicts of interest.

\end{document}